\documentclass{amsart}
\usepackage{amsmath,amsfonts,color}

\usepackage{graphicx}

\DeclareGraphicsExtensions{.pdf,.png,.jpg}
\usepackage{xcolor}
\usepackage{tikz}
\usetikzlibrary{arrows}

\date{July 24, 2012}

\newtheorem{Theo}{Theorem}
\newtheorem{Lem}{Lemma}
\newtheorem{Cor}{Corollary}
\newcommand{\multi}{\left(1-\frac{z^2}{R^2}\right)^\delta}

\newcommand{\Bes}{J_{\nu}}

\renewcommand{\Im}{\mathop{\rm Im}}

\begin{document} \title[The Bochner-Riesz means for Fourier-Bessel expansions]
{The Bochner-Riesz means for Fourier-Bessel expansions: norm inequalities
for the maximal operator and almost everywhere convergence}

\author[\'O. Ciaurri and L. Roncal]{\'Oscar Ciaurri and Luz Roncal}
\address{Departamento de Matem\'aticas y Computaci\'on\\
Universidad de La Rioja\\ 26004 Logro\~no, Spain}
\email{oscar.ciaurri@unirioja.es, luz.roncal@unirioja.es}

\thanks{Research supported by the grant MTM2009-12740-C03-03 from Spanish Government.}



\begin{abstract}
In this paper, we develop a thorough analysis of the boundedness properties
of the maximal operator for the Bochner-Riesz means related to the Fourier-Bessel
expansions. For this operator, we study weighted and unweighted inequalities in the
spaces $L^p((0,1),x^{2\nu+1}\, dx)$. Moreover, weak and restricted weak type
inequalities are obtained for the critical values of~$p$. As a consequence,
we deduce the almost everywhere pointwise convergence of these means.
\end{abstract}

\maketitle

\section{Introduction and main results}

Let $\Bes$ be the Bessel function of order $\nu$. For $\nu>-1$ we
have that
\[
\int_0^1 \Bes(s_jx)\Bes(s_kx)x\,dx= \frac12
(J_{\nu+1}(s_j))^2\delta_{j,k},\quad j,k=1,2,\dots
\]
where $\{s_j\}_{j\ge 1}$ denotes the sequence of successive
positive zeros of $\Bes$. From the previous identity we can check
that the system of functions
\begin{equation}
\label{eq:FBesselSystemI}
\psi_j(x)=\frac{\sqrt{2}}{|J_{\nu+1}(s_j)|}x^{-\nu}\Bes(s_jx),\quad
j=1,2,\dots
\end{equation}
is orthonormal and complete in $L^2((0,1),d\mu_\nu)$, with
$d\mu_\nu(x)=x^{2\nu+1}\, dx$ (for the completeness,
see~\cite{Hochstadt}). Given a function $f$ on $(0,1)$, its
Fourier series associated with this system, named as
Fourier-Bessel series, is defined by
\begin{equation}\label{SeriesCoeficientes}
f\sim \sum_{j=1}^\infty a_j(f)\psi_j,\qquad\text{with}\qquad
a_j(f)=\int_0^1 f(y)\psi_j(y)\,d\mu_{\nu}(y),
\end{equation}
provided the integral exists. When $\nu=n/2-1$, for $n\in
\mathbb{N}$ and $n\ge 2$, the functions $\psi_j$ are the
eigenfunctions of the radial Laplacian in the multidimensional
ball $B^n$. The eigenvalues are the elements of the sequence
$\{s_j^2\}_{j\ge 1}$. The Fourier-Bessel series corresponds with
the radial case of the multidimensional Fourier-Bessel expansions
analyzed in~\cite{Bal-Cor}.

For each $\delta>0$, we define the Bochner-Riesz means for
Fourier-Bessel series as
\begin{equation*}
\mathcal{B}_R^{\delta}(f,x)=\sum_{j\ge 1}
\left(1-\frac{s_j^2}{R^2}\right)_+^{\delta}a_j(f)\psi_j(x),
\end{equation*}
where $R>0$ and $(1-s^2)_+=\max\{1-s^2,0\}$. Bochner-Riesz means
are a regular summation method used oftenly in harmonic analysis.
It is very common to analyze regular summation methods for Fourier
series when the convergence of the partial sum fails. Ces\`{a}ro
means are other of the most usual summation methods. B.
Muckenhoupt and D. W. Webb \cite{Mu-We} give inequalities for
Ces\`{a}ro means of Laguerre polynomial series and for the
supremum of these means with certain parameters and $1<p\leq
\infty$. For $p=1$, they prove a weak type result. They also
obtain similar estimates for Ces\`{a}ro means of Hermite
polynomial series and for the supremum of those means in
\cite{Mu-We-Her}. An almost everywhere convergence result is
obtained as a corollary in \cite{Mu-We} and \cite{Mu-We-Her}. The
result about Laguerre polynomials is an extension of a previous
result in \cite{Stem}. This kind of matters has been also studied
by the first author and J. L. Varona in \cite{Ciau-Var} for the
Ces\`{a}ro means of generalized Hermite expansions. The Ces\`{a}ro
means for Jacobi polynomials were analyzed by S. Chanillo and B.
Muckenhoupt in \cite{Ch-Muc}. The Bochner-Riesz means themselves
have been analyzed for the Fourier transform and their boundedness
properties in $L^p(\mathbb{R}^n)$ is an important unsolved problem
for $n>2$ (the case $n=2$ is well understood, see \cite{Car-Sjo}).

The target of this paper is twofold. First we will analyze the
almost everywhere (a. e.) convergence, for functions in
$L^p((0,1),d\mu_\nu)$, of the Bochner-Riesz means for
Fourier-Bessel expansions. By the general theory \cite[Ch.
2]{Duoa}, to obtain this result we need to estimate the maximal
operator
\[
\mathcal{B}^{\delta}(f,x)=\sup_{R>0}\left|\mathcal{B}_R^{\delta}(f,x)\right|,
\]
in the $L^p((0,1),d\mu_\nu)$ spaces. A deep analysis of the
boundedness properties of this operator will be the second goal of
our paper. This part of our work is strongly inspired by the
results given in \cite{Ch-Muc} for the Fourier-Jacobi expansions.

Before giving our results we introduce some notation. Being
$p_0=\frac{4(\nu+1)}{2\nu+3+2\delta}$ and
$p_1=\frac{4(\nu+1)}{2\nu+1-2\delta}$, we define
\begin{align}
\label{eq:endpoints0} p_0(\delta)&=\begin{cases}
1,& \delta> \nu+1/2\text{ or } -1<\nu\le -1/2,\\
p_0,& \delta\le \nu+1/2\text{ and } \nu>-1/2,
\end{cases}\\
\notag p_1(\delta)&=\begin{cases}
\infty,& \delta> \nu+1/2\text{ or } -1<\nu\le -1/2,\\
p_1,& \delta\le \nu+1/2\text{ and } \nu>-1/2.
\end{cases}
\end{align}

Concerning to the a. e. convergence of the Bochner-Riesz means,
our result reads as follows
\begin{Theo}
\label{th:fin} Let $\nu>-1$, $\delta>0$, and $1\le p<\infty$.
Then,
\begin{equation*}
\mathcal{B}_R^\delta(f,x)\to f(x)\quad \text{a.
e., for $f\in L^p((0,1),d\mu_\nu)$}
\end{equation*}
if and only if $p_0(\delta)\le p$, where $p_0(\delta)$ is as in
\eqref{eq:endpoints0}.
\end{Theo}
Proof of Theorem \ref{th:fin} is contained in Section
\ref{sec:Proof Th1} and is based on the following arguments. On
one hand, to prove the necessity part, we will show the existence
of functions in $L^p((0,1),d\mu_{\nu})$ for $p<p_0(\delta)$ such
that $\mathcal{B}^\delta_R$ diverges for them. In order to do
this, we will use a reasoning similar to the one given by C.
Meaney in \cite{Meaney} that we describe in Section \ref{sec:Proof
Th1}. On the other hand, for the sufficiency, observe that the
convergence result follows from the study of the maximal operator
$\mathcal{B}^\delta f$. Indeed, it is sufficient to get
$(p_0(\delta),p_0(\delta))$-weak type estimates for this operator
and this will be the content of Theorem
\ref{th:AcDebilMaxRedonda}.

Regarding the boundedness properties of $\mathcal{B}^\delta f$ we
have the following facts. First, a result containing the
$(p,p)$-strong type inequality.
\begin{Theo}
\label{th:max} Let $\nu>-1$, $\delta>0$, and $1< p\le\infty$.
Then,
\begin{equation*}
\left\|\mathcal{B}^{\delta}f\right\|_{L^p((0,1),d\mu_{\nu})}\le C
\|f\|_{L^p((0,1),d\mu_{\nu})}
\end{equation*}
if and only if
\[
\begin{cases}
1<p\le \infty, &\text{for $-1<\nu\le -1/2$ or $\delta>\nu+1/2$},\\
p_0<p<p_1, &\text{for $\delta\le \nu+1/2$ and $\nu>-1/2$.}
\end{cases}
\]
\end{Theo}
In the lower critical value of $p_0(\delta)$ we can prove a
$(p_0(\delta),p_0(\delta))$-weak type estimate.
\begin{Theo}
\label{th:AcDebilMaxRedonda} Let $\nu>-1$, $\delta>0$, and
$p_0(\delta)$ be the number in \eqref{eq:endpoints0}. Then,
\[
\left\|\mathcal{B}^{\delta}f\right\|_{L^{p_0(\delta),\infty}((0,1),d\mu_{\nu})}\le
C \|f\|_{L^{p_0(\delta)}((0,1),d\mu_{\nu})},
\]
with $C$ independent of $f$.
\end{Theo}
Finally, for the upper critical value, when $0<\delta<\nu+1/2$ and
$\nu>-1/2$, it is possible to obtain a $(p_1,p_1)$-restricted weak
type estimate.
\begin{Theo}
\label{th:AcDebilRestMaxRedonda} Let $\nu>-1/2$ and
$0<\delta<\nu+1/2$. Then,
\[
\left\|\mathcal{B}^{\delta}\chi_E\right\|_{L^{p_1,\infty}((0,1),d\mu_{\nu})}\le
C \|\chi_E\|_{L^{p_1}((0,1),d\mu_{\nu})},
\]
for all measurable subsets $E$ of $(0,1)$ and $C$ independent of
$E$.
\end{Theo}
The previous results about norm inequalities are summarized in Figure 1 (case $-1<\nu \le -1/2$) and Figure 2 (case $\nu>-1/2$).

\begin{center}
\begin{tikzpicture}[scale=2.45]
\fill[black!10!white] (0.1,0.1) -- (2.1,0.1) -- (2.1,2.1) -- (0.1,2.1) -- cycle;
\draw[thick, dashed] (2.1,0.1) -- (2.1,2.1);
\draw[very thick] (0.1,2.1) -- (0.1,0.1);
\draw[very thin] (0.1,0.1) -- (2.15,0.1);
\node at (0.1,0) {$0$};
\node at (2.1,0) {$1$};
\node at (2.2,0.1) {$\frac{1}{p}$};
\node at (0.1,2.175) {$\delta$};
\draw (0,1.4) node [rotate=90]  {\tiny{$(p,p)$-strong}};
\draw (2.2,1.4) node [rotate=90]  {\tiny{$(p,p)$-weak}};
\draw (1.1,-0.2) node {Figure 1: case $-1<\nu\le-\tfrac{1}{2}$.};

\fill[black!10!white] (2.7,0.75) -- (3.35,0.1) -- (4.05,0.1) -- (4.7,0.75) -- (4.7,2.1) -- (2.7,2.1) -- cycle;
\draw[thick, dashed] (4.7,0.75) -- (4.7,2.1);
\draw[very thick] (2.7,2.1) -- (2.7,0.78);
\filldraw[fill=white] (2.7,0.75) circle (0.8pt);
\draw[thick, dotted] (2.72,0.73) -- (3.35,0.1);
\draw[thick, dashed] (4.05,0.1) -- (4.7,0.75);
\draw[very thin] (2.7,0.1) -- (4.75,0.1);
\draw[very thin] (2.7,0.1) -- (2.7,0.72);
\draw[very thin] (4.7,0.1) -- (4.7,0.13);
\node at (2.7,0) {$0$};
\node at (4.7,0) {$1$};
\node at (4.8,0.1) {$\frac{1}{p}$};
\node at (2.7,2.175) {$\delta$};
\draw (2.43,0.75) node {\tiny{$\delta=\nu+\tfrac{1}{2}$}};
\draw (3.35,0) node {\tiny{$\tfrac{2\nu+1}{4(\nu+1)}$}};
\draw (4.05,0) node {\tiny{$\tfrac{2\nu+3}{4(\nu+1)}$}};
\draw (2.6,1.4) node [rotate=90]  {\tiny{$(p,p)$-strong}};
\draw (3.1,0.45) node [rotate=-45]  {\tiny{$(p,p)$-restric. weak}};
\draw (4.3,0.45) node [rotate=45]  {\tiny{$(p,p)$-weak}};
\draw (4.8,1.4) node [rotate=90]  {\tiny{$(p,p)$-weak}};
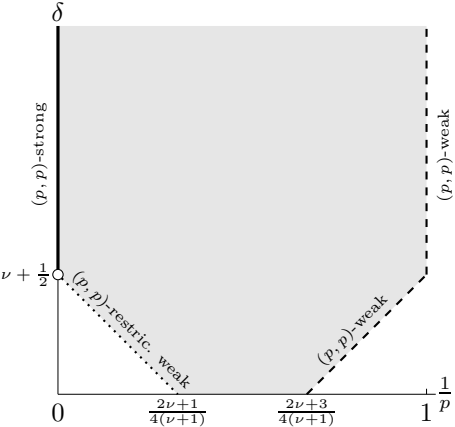
\draw (3.7,-0.2) node {Figure 2: case $\nu>-\tfrac{1}{2}$.};
\end{tikzpicture}
\end{center}


At this point, a comment is in order. Note that J. E.
Gilbert \cite{Gi} also proves weak type norm inequalities for
maximal operators associated with orthogonal expansions. The method used
cannot be applied in our case, and the reason is the same as can
be read in \cite{Ch-Muc}, at the end of Sections 15 and 16
therein. Following the technique in \cite{Gi} we have to analyze
some weak type inequalities for Hardy operator and its adjoint with weights and these
inequalities do not hold for $p=p_0$ and $p=p_1$.

The proof of the sufficiency in Theorem \ref{th:max} will be
deduced from a more general result in which we analyze the
boundedness of the operator $\mathcal{B}^\delta f$ with potential
weights. Before stating it, we need a previous definition. We say
that the parameters $(b,B,\nu,\delta)$ satisfy the $C_p$
conditions if
\begin{align}
 b& > \frac{-2(\nu+1)}{p} \,\,\, (\ge \text{ if }p=\infty), \label{ec:con1B}\\
 B& < 2(\nu+1)\left(1-\frac1p\right) \,\,\, (\le \text{ if }
 p=1), \label{ec:con2B}\\
 b& > 2(\nu+1)\left(\frac12-\frac1p\right)-\delta-\frac12\,\,\,
 (\ge \text{ if }p=\infty), \label{ec:con3B}\\
 B& \le 2(\nu+1)\left(\frac12-\frac1p\right)+\delta+\frac12, \label{ec:con4B}\\
 B &\le b \label{ec:con5B},
\end{align}
and in at least one of each of the following pairs the inequality
is strict: \eqref{ec:con2B} and \eqref{ec:con5B}, \eqref{ec:con3B}
and \eqref{ec:con5B}, and \eqref{ec:con4B} and \eqref{ec:con5B}
except for $p=\infty$. The result concerning inequalities with
potential weights is the following.
\begin{Theo}
\label{th:AcFuerteMaxRedonda} Let $\nu>-1$, $\delta>0$, and $1<
p\le\infty$. If $(b,B,\nu,\delta)$ satisfy the $C_p$ conditions,
then
\begin{equation*}
\left\|x^b\mathcal{B}^{\delta}f\right\|_{L^p((0,1),d\mu_{\nu})}\le
C \|x^Bf\|_{L^p((0,1),d\mu_{\nu})},
\end{equation*}
with $C$ independent of $f$.
\end{Theo}

A result similar to Theorem \ref{th:AcFuerteMaxRedonda} for the
partial sum operator was proved in \cite [Theorem 1]{GuPeRuVa}. It
followed from a weighted version of a general Gilbert's maximal
transference theorem, see \cite[Theorem 1]{Gi}. The weighted
extension of Gilbert's result given in \cite{GuPeRuVa} depended
heavily on the $A_p$ theory and it can not be used in our case
because it did not capture all the information relative to the
weights. On the other hand, it is also remarkable the paper by K.
Stempak \cite{Stem2} in which maximal inequalities for the partial
sum operator of Fourier-Bessel expansions and divergence and
convergence results are discussed.

The necessity in Theorem \ref{th:max} will follow by showing that
the operator $\mathcal{B}^\delta f$ is neither $(p_1,p_1)$-weak
nor $(p_0,p_0)$-strong for $\nu>-1/2$ and $0<\delta\le \nu+1/2$.
This is the content of the next theorems.
\begin{Theo}
\label{th:noweak} Let $\nu>-1/2$. Then
\[
\sup_{\|f\|_{L^{p_1}((0,1),d\mu_\nu)}=1}
\|\mathcal{B}^{\delta}_{R}f\|_{L^{p_1,\infty}((0,1),d\mu_\nu)}\ge
C (\log R)^{1/p_0},
\]
if $0<\delta<\nu+1/2$; and
\[
\sup_{\|f\|_{L^\infty((0,1),d\mu_\nu)}=1}
\|\mathcal{B}^{\delta}_{R}f\|_{L^{\infty}((0,1),d\mu_\nu)}\ge C
\log R,
\]
if $\delta= \nu+1/2$.
\end{Theo}

\begin{Theo}
\label{th:nostrong} Let $\nu>-1/2$. Then
\[
\sup_{E\subset
(0,1)}\frac{\|\mathcal{B}^{\delta}_{R}\chi_E\|_{L^{p_0}((0,1),d\mu_\nu)}}{\|\chi_E\|_{L^{p_0}((0,1),d\mu_\nu)}}\ge
C (\log R)^{1/p_0},
\]
if $0<\delta<\nu+1/2$; and
\[
\sup_{\|f\|_{L^1((0,1),d\mu_\nu)}=1}
\|\mathcal{B}^{\delta}_{R}f\|_{L^{1}((0,1),d\mu_\nu)}\ge C \log R,
\]
if $\delta= \nu+1/2$.
\end{Theo}

The paper is organized as follows. In the next section, we give
the proof of Theorem \ref{th:fin}. In Section
\ref{sec:proofAcFuerte} we first relate the Bochner-Riesz means
$\mathcal{B}_R^{\delta}$ to the Bochner-Riesz means operator
associated with the Fourier-Bessel system in the Lebesgue measure
setting. Then, we prove weighted inequalities for the supremum of
this new operator. With the connection between these means and the
operator $\mathcal{B}_R^{\delta}$, we obtain Theorem
\ref{th:AcFuerteMaxRedonda} and, as a consequence, the sufficiency
of Theorem \ref{th:max}. Sections
\ref{sec:ProofThAcDebilMaxRedonda} and \ref{sec:acdelrest} will be
devoted to the proofs of Theorems \ref{th:AcDebilMaxRedonda} and
\ref{th:AcDebilRestMaxRedonda}, respectively. The proofs of
Theorems \ref{th:noweak} and \ref{th:nostrong} are contained in
Section \ref{sec:negativeths}. One of the main ingredients in the
proofs of Theorems \ref{th:noweak} and \ref{th:nostrong} will be
Lemma \ref{lem:pol}, this lemma is rather technical and it will be
proved in the Section \ref{sec:techlemma}.

Throughout the paper, we will use the following notation: for each
$p\in[1,\infty]$, we will denote by $p'$ the conjugate of $p$,
that is, $\tfrac{1}{p}+\tfrac{1}{p'}=1$. We shall write $X\simeq Y$
when simultaneously $X\le C Y$ and $Y \le C X$.

\section{Proof of Theorem \ref{th:fin}}\label{sec:Proof Th1}

The proof of the sufficiency follows from Theorem
\ref{th:AcDebilMaxRedonda} and standard arguments.

In order to prove the necessity, let us see that, for
$0<\delta<\nu+1/2$ and $\nu>-1/2$, there exists a function $f\in
L^{p}((0,1),d\mu_{\nu})$, $p\in [1,p_0)$, for which
$\mathcal{B}_R^{\delta}(f,x)$ diverges. We follow some ideas
contained in \cite{Meaney} and \cite{Stem2}.

First, we need a few more ingredients. Recall the well-known
asymptotics for the Bessel functions (see \cite[Chapter 7]{Wat})
\begin{equation}\label{zero}
J_\nu(z) = \frac{z^\nu}{2^\nu \Gamma(\nu+1)} + O(z^{\nu+2}),
    \quad |z|<1,\quad |\arg(z)|\leq\pi,
\end{equation}
and
\begin{equation}
\label{infty} J_\nu(z)=\sqrt{\frac{2}{\pi z}}\left[
\cos\left(z-\frac{\nu\pi}2
    - \frac\pi4 \right) + O(e^{\Im(z)}z^{-1}) \right], \quad |z|
    \ge 1,\quad |\arg(z)|\leq\pi-\theta,
\end{equation}
where $D_{\nu}=-(\nu\pi/2+\pi/4)$. It will also be useful the fact
that (cf. \cite[(2.6)]{OScon})
\begin{equation}
\label{eq:zerosCons} s_{j}=O(j).
\end{equation}
For our purposes, we need estimates for the $L^p$ norms of the
functions $\psi_j$. These estimates are contained in the following
lemma, whose proof can be read in \cite[Lemma 2.1]{Ci-RoWave}.
\begin{Lem}
\label{Lem:NormaFunc} Let $1\le p\le\infty$ and $\nu>-1$. Then,
for $\nu>-1/2$,
\begin{equation*}
\|\psi_j\|_{L^p((0,1),d\mu_{\nu})}\simeq
\begin{cases} j^{(\nu+1/2)-\frac{2(\nu+1)}{p}}, &
\text{if $p>\frac{2(\nu+1)}{\nu+1/2}$},\\
(\log j)^{1/p}, & \text{if $p=\frac{2(\nu+1)}{\nu+1/2}$},\\
1, & \text{if $p<\frac{2(\nu+1)}{\nu+1/2}$},
\end{cases}
\end{equation*}
and, for $-1<\nu\le-1/2$,
\begin{equation*}
\|\psi_j\|_{L^p((0,1),d\mu_{\nu})}\simeq
\begin{cases}
1, & \text{if $p<\infty$},\\
j^{\nu+1/2},& \text{if $p=\infty$}.
\end{cases}
\end{equation*}
\end{Lem}
We will also use a slight modification of a result by G. H. Hardy
and M. Riesz for the Riesz means of order $\delta$, that is
contained in \cite[Theorem 21]{HaRi}. We present here this
 result, adapted to the Bochner-Riesz means. We denote by $S_R(f,x)$
the partial sum associated to the Fourier-Bessel expansion, namely
\[
S_R(f,x)=\sum_{0<s_j\le R} a_j(f)\psi_j(x).
\]
The result reads as follows.
\begin{Lem}
\label{lem:HardyRiesz} Suppose that $f$ can be expressed as a
Fourier-Bessel expansion and for some $\delta>0$ and $x\in(0,1)$
its Bochner-Riesz means $\mathcal{B}_R^\delta(f,x)$ converges to
$c$ as $R\rightarrow \infty$. Then, for $s_n\le R < s_{n+1}$,
\[
|S_R(f,x)-c|\le A_{\delta}n^{\delta}\sup_{0<t\le
s_{n+1}}|\mathcal{B}_{t}^{\delta}(f,x)|.
\]
\end{Lem}
By using this lemma, we can write
\begin{equation}
\label{ConsecuenciaLemaHardyRiesz}
|a_j(f)\psi_j(x)|=|(S_{s_{j}}(f,x)-c)-(S_{s_{j-1}}(f,x)-c)|\le
A_{\delta}j^\delta\sup_{0<t\le
s_{j+1}}|\mathcal{B}_{t}^{\delta}(f,x)|.
\end{equation}
Let us proceed with the proof of the necessity. Let $1\le p<p_0$.
Note that $p_0'=p_1$. Therefore,
$p'>p_0'>\tfrac{2(\nu+1)}{\nu+1/2}$, and
$\delta<\nu+1/2-\frac{2(\nu+1)}{p'}:=\lambda$. By Lemma
\ref{Lem:NormaFunc}, $\|\psi_j\|_{L^{p'}((0,1),d\mu_{\nu})}\ge C
j^{\lambda}$. Then, we have that the mapping $f\mapsto a_j(f)$,
where $a_j(f)$ was given in \eqref{SeriesCoeficientes}, is a
bounded linear functional on $L^{p}((0,1),d\mu_{\nu})$ with norm
bounded below by a constant multiple of $j^\lambda$. By uniform
boundedness principle, for $p$ conjugate to $p'$ and each $0\le
\varepsilon<\lambda$, there is a function $f_0\in L^p((0,1),
d\mu_{\nu})$ so that $a_j(f_0)j^{-\varepsilon}\rightarrow \infty$
as $j\rightarrow \infty$. By taking $\varepsilon=\delta$, we have
that
\begin{equation}\label{coeficienteInfinito}
a_j(f_0)j^{-\delta}\rightarrow \infty \quad \textrm{ as } \quad
j\rightarrow \infty.
\end{equation}

Suppose now that $B_R^\delta(f_0,x)$ converges. Then, by Egoroff's
theorem, it converges on a subset $E$ of positive measure in
$(0,1)$ and, clearly, we can think that $E\subset (\eta, 1)$ for
some fixed $\eta>0$. For each $x\in E$, we can consider $j$ such
that $s_j x\ge 1$ and, by \eqref{infty},
\begin{align*}
    |a_j(f_0)\psi_j(x)|&=\big|a_j(f_0)
    \Big(\frac{\sqrt2}{|J_{\nu+1}(s_j)|}
    x^{-\nu}J_{\nu}(s_jx)\\
    &-\frac{\sqrt2}{|J_{\nu+1}(s_j)|}x^{-\nu}
    \Big(\frac{2}{\pi s_jx}\Big)^{1/2}
    \cos(s_jx+D_{\nu})\Big)\\
    &+a_j(f_0) \frac{\sqrt2}{|J_{\nu+1}(s_j)|}x^{-\nu}
    \Big(\frac{2}{\pi s_jx}\Big)^{1/2}
    \cos(s_jx+D_{\nu})\big|\\
    &=Cs_j^{-1/2}\frac{\sqrt 2}{|J_{\nu+1}(s_j)|}|a_j(f_0)x^{-\nu-1/2}
    \big(O((s_jx)^{-1})+\cos(s_jx+D_{\nu})\big)|\\
    &\simeq |a_j(f_0)x^{-\nu-1/2}(\cos(s_jx+D_{\nu})+O((s_jx)^{-1}))|.
\end{align*}
By \eqref{ConsecuenciaLemaHardyRiesz} on this set $E$,
\[
|a_j(f_0)x^{-\nu-1/2}(\cos(s_jx+D_{\nu})+O((j)^{-1}))|\le
A_{\delta}j^{\delta}\sup_{0<t\le
s_{j+1}}|\mathcal{B}_{t}^{\delta}(f_0,x)|\le K_Ej^{\delta},
\]
uniformly on $x\in E$. We also used \eqref{eq:zerosCons} in the
latter. The inequality above is equivalent to
$$
|a_j(f_0)(\cos(s_jx+D_{\nu})+O(j^{-1}))|\le K_E
x^{\nu+1/2}j^{\delta}\le K_{E}j^{\delta}.
$$
Therefore,
\begin{equation}
\label{ec:boundFj}
|a_j(f_0)j^{-\delta}(\cos(s_jx+D_{\nu})+O((j)^{-1}))|\le K_E.
\end{equation}
Now, taking the functions
\[
F_j(x)=a_j(f_0)j^{-\delta}(\cos(s_jx+D_{\nu})+O(j^{-1})), \qquad
x\in E,
\]
and using an argument based on the Cantor-Lebesgue and
Riemann-Lebesgue theorems, see \cite[Section 1.5]{Meaney} and
\cite[Section IX.1]{Zyg}, we obtain that
\[
\int_E |F_j(x)|^2\, dx\ge C |a_j(f_0)j^{-\delta}|^2|E|,
\]
where, as usual, $|E|$ denotes the Lebesgue measure of the set
$E$. On the other hand, by \eqref{ec:boundFj},
\[
\int_E |F_j(x)|^2\, dx\le K_E^2 |E|.
\]
Then, from the previous estimates, it follows that
$|a_j(f_0)j^{-\delta}|\le C$, which contradicts
\eqref{coeficienteInfinito}.
\section{Bochner-Riesz means for Fourier-Bessel expansions in the Lebesgue
measure setting. Proof of Theorem
\ref{th:AcFuerteMaxRedonda}}\label{sec:proofAcFuerte}

For our convenience, we are going to introduce a new orthonormal
system. We will take the functions
\[
\phi_j(x)=\frac{\sqrt{2x}\Bes(s_jx)}{|J_{\nu+1}(s_j)|},\quad
j=1,2,\dots.
\]
These functions are a slight modification of the functions
\eqref{eq:FBesselSystemI}; in fact,
\begin{equation}
\label{eq:Relation}
\phi_j(x)=x^{\nu+1/2}\psi_j(x).
\end{equation}
The system $\{\phi_j(x)\}_{j\ge1}$ is a complete orthonormal basis
of $L^2((0,1),dx)$.

In this case, the corresponding Fourier-Bessel expansion of a
function $f$ is
\[
f\sim\sum_{j=1}^{\infty}b_j(f) \phi_j(x), \qquad \text{with}
\qquad b_j(f)=\left(\int_0^1 f(y)\phi_j(y)\, dy\right)
\]
provided the integral exists, and for $\delta>0$ the Bochner-Riesz
means of this expansion are
\[
B_R^{\delta}(f,x)=\sum_{j\ge 1}
\left(1-\frac{s_j^2}{R^2}\right)_+^{\delta}b_j(f)\phi_j(x),
\]
where $R>0$ and $(1-s^2)_+=\max\{1-s^2,0\}$. It follows that
\[
B_R^{\delta}(f,x)=\int_0^1 f(y)K_R^\delta(x,y)\, dy
\]
where
\begin{equation}
\label{ec:kern} K_R^\delta(x,y)=\sum_{j\ge
1}\left(1-\frac{s_j^2}{R^2}\right)_+^{\delta}\phi_j(x)\phi_j(y).
\end{equation}

Our next target is the proof of Theorem
\ref{th:AcFuerteMaxRedonda}. Taking into account that
\[
\mathcal{B}_R^\delta f(x)=\int_0^1 f(y)\mathcal{K}_R^\delta(x,y)
\, d\mu_\nu(y),
\]
where
\[
\mathcal{K}_R^\delta(x,y)=\sum_{j\ge
1}\left(1-\frac{s_j^2}{R^2}\right)_+^{\delta}\psi_j(x)\psi_j(y),
\]
it is clear, from \eqref{eq:Relation}, that
$\mathcal{K}_R^\delta(x,y)=(xy)^{-(\nu+1/2)}K_R^{\delta}(x,y)$.
Then, it is verified that the inequality
\[
\|x^b\mathcal{B}^{\delta}(f,x)\|_{L^p((0,1),d\mu_{\nu})}\le
C\|x^Bf(x)\|_{L^p((0,1),d\mu_{\nu})}
\]
is equivalent to
\begin{equation*}
\|x^{b+(\nu+1/2)(2/p-1)}B^{\delta}(f,x)\|_{L^p((0,1),dx)}\le
C\|x^{B+(\nu+1/2)(2/p-1)}f(x)\|_{L^p((0,1),dx)},
\end{equation*}
that is, we can focus on the study of a weighted inequality for
the operator $B_R^{\delta}(f,x)$. The first results about
convergence of this operator can be found in \cite{Ci-Ro}.

We are going to prove an inequality of the form
\begin{equation*}
\|x^{a}B^{\delta}(f,x)\|_{L^p((0,1),dx)}\le C
\|x^{A}f(x)\|_{L^p((0,1),dx)}
\end{equation*}
for $\delta>0$, $1< p\leq \infty$, under certain conditions for
$a, A,\nu$ and $\delta$. Besides, a weighted weak type result for
$\sup_{R>0}|B_R^{\delta}(f,x)|$ will be proved for $p=1$. The
abovementioned conditions are the following. Let $\nu>-1$,
$\delta>0$ and $1\leq p\leq \infty$; parameters $(a,A,\nu,\delta)$
will be said to satisfy the $c_p$ conditions provided
\begin{align}
  a & > -1/p-(\nu+1/2) \,\,\, (\ge \text{ if } p=\infty), \label{ec:con1}\\
  A & < 1-1/p+(\nu+1/2)\,\,\, (\le \text{ if } p=1),\label{ec:con2}\\
  a &> -\delta-1/p\,\,\, (\ge \text{ if }p=\infty),\label{ec:con3}\\
  A &\le  1+\delta-1/p, \label{ec:con4}\\
  A &\le a\label{ec:con5}
\end{align}
and in at least one of each of the following pairs the inequality
is strict: \eqref{ec:con2} and \eqref{ec:con5}, \eqref{ec:con3}
and \eqref{ec:con5}, and \eqref{ec:con4} and \eqref{ec:con5}
except for $p=\infty$.

The main results in this section are the following:
\begin{Theo}
\label{th:main1} Let $\nu>-1$, $\delta>0$ and $1< p\le \infty$. If
$(a, A, \nu, \delta)$ satisfy the $c_p$ conditions, then
\[
\|x^{a}B^{\delta}(f,x)\|_{L^p((0,1),dx)}\le C
\|x^{A}f(x)\|_{L^p((0,1),dx)},
\]
with $C$ independent of $f$.
\end{Theo}

\begin{Theo}
\label{th:main2} Let $\nu>-1$ and $\delta>0$. If $(a, A, \nu,
\delta)$ satisfy the $c_1$ conditions and
\[
E_{\lambda}=\left\{x\in (0,1)\colon x^{a}
\sup_{R>0}\left(|B_R^{\delta}(f,x)|\right)>\lambda \right\},
\]
 then
\[
|E_{\lambda}|\leq C \frac{\|x^{A}
f(x)\|_{L^1((0,1),dx)}}{\lambda},
\]
with $C$ independent of $f$ and $\lambda$.
\end{Theo}

Note that, taking $a=b+(\nu+1/2)(2/p-1)$ and
$A=B+(\nu+1/2)(2/p-1)$, Theorem \ref{th:AcFuerteMaxRedonda}
follows from Theorem \ref{th:main1}.

The proofs of Theorem \ref{th:main1} and Theorem \ref{th:main2}
will be achieved by decomposing the square $(0,1)\times (0,1)$
into five regions and obtaining the estimates therein. The regions
will be:
\begin{align}
\label{regions}
\notag A_1&=\{(x,y):0 < x, y\leq  4/R\},\\
\notag A_2&=\{(x,y):4/R<\max\{x,y\}<1,\, |x-y|\le 2/R  \},\\
 A_3&=\{(x,y): 4/R \leq x < 1,\, 0 < y\leq x/2\},\\
\notag A_4&=\{(x,y):0 < x \leq y/2,\, 4/R \leq y< 1\}, \\
\notag A_5&=\{(x,y): 4/R < x <  1, \, x/2 < y< x- 2/R\}\\
\notag &\kern25pt\cup
\{(x,y): y/2 < x \leq y-2/R,\, 4/R \leq y<1\}.
\end{align}

Theorem~\ref{th:main1} and Theorem~\ref{th:main2} will follow by
showing that, if $1\leq p\leq \infty$, then
\begin{equation}
\label{eq:des_1} \left\|\sup_{R> 0}\int_0^1
y^{-A}x^a|K_R^\delta(x,y)||f(y)|\chi_{A_j}\,dy\right\|_{L^p((0,1),dx)}
\leq C\|f(x)\|_{L^p((0,1),dx)}
\end{equation}
holds for $j=1,3,4$ and that
\begin{equation}
\label{eq:des_2} \int_0^1
y^{-A}x^a|K_R^\delta(x,y)||f(y)|\chi_{A_j}\,dy \leq C M (f,x),
\end{equation}
for $j=2,5$, where $M$ is the Hardy-Littlewood maximal function of
$f$, and $C$ is independent of $R, x$ and $f$. These results and
the fact that $M$ is $(1,1)$-weak and $(p,p)$-strong if $1<p\leq
\infty$ complete the proofs.

To get \eqref{eq:des_1} and \eqref{eq:des_2} we will use a very
precise pointwise estimate for the kernel $K_R^\delta(x,y)$,
obtained in \cite{Ci-Ro}; there, it was shown that
\begin{equation}
\label{ec:kernel}
|K_R^\delta(x,y)|\le C \begin{cases} (xy)^{\nu+1/2}R^{2(\nu+1)}, &
(x,y) \in A_1,\\ R, & (x,y) \in A_2\\
\frac{\Phi_\nu(Rx)\Phi_{\nu}(Ry)}{R^{\delta}|x-y|^{\delta+1}}, &
(x,y) \in A_3\cup A_4 \cup A_5, \end{cases}
\end{equation}
with
\begin{equation}
\label{ec:aux}
\Phi_\nu(t)=\begin{cases}t^{\nu+1/2}, & \text{ if $0<t<2$},\\
1,& \text{ if $t\ge 2$}.\end{cases}
\end{equation}

The proof of \eqref{eq:des_2} follows from the given estimate for
the kernel $K_R^\delta(x,y)$ and $y^{-A}x^a\simeq C$ in $A_2\cup
A_5$ because $A\le a$. In the case of $A_2$, from
$|K_R^\delta(x,y)|\le C R$ we deduce easily the required
inequality. For $A_5$ the result is a consequence of
$\Phi_{\nu}(Rx)\Phi_\nu(Ry)\le C$ and of a decomposition of the
region in strips such that $R|x-y|\simeq 2^{k}$, with $k=0,\dots,
[\log_2 R]-1$; this can be seen in \cite[p. 109]{Ci-Ro}

In this manner, to complete the proofs of Theorem~\ref{th:main1}
and Theorem~\ref{th:main2} we only have to show \eqref{eq:des_1}
for $j=1,3,4$ in the conditions $c_p$ for $1\le p \le \infty$, and
this is the content of Corollary~\ref{cor:corolario2} in
Subsection~\ref{subsec:reg1}. In its turn,
Corollary~\ref{cor:corolario2} follows from Lemmas~\ref{lem:lema7}
and~\ref{lem:lema8} in the same subsection. Previously, Subsection
\ref{subsec:lemmas} contains some technical lemmas that will be
used in the proofs of Lemmas~\ref{lem:lema7} and~\ref{lem:lema8}.

\subsection{Technical Lemmas}
\label{subsec:lemmas}
To prove \eqref{eq:des_1} for $j=1,3,4$ we will use an
interpolation argument based on six lemmas. These are stated
below. They are small modifications of the six lemmas contained in
Section 3 of \cite{Mu-We} where a sketch of their proofs can be
found.

\begin{Lem}
\label{lem:lema1} Let $\xi_0>0$, if $r<-1$, $r+t\leq-1$ and
$r+s+t\leq-1$, then for $p=1$
\[
\left\|x^r\chi_{[1,\infty)}(x) \sup_{\xi_0\leq\xi\leq x}\xi^s
\int_\xi^x y^t|f(y)|\,dy\right\|_{L^p((0,\infty),dx)}\leq
C\|f(x)\|_{L^p((0,\infty),dx)}
\]
with $C$ independent of $f$. If $r\leq0$, $r+t\leq-1$ and
$r+s+t\leq-1$ with equality holding in at most one of the first
two inequalities, then this holds for $p=\infty$.
\end{Lem}

\begin{Lem}
\label{lem:lema2} Let $\xi_0>0$, if $t\leq0$, $r+t\leq-1$ and
$r+s+t\leq-1$, with strict inequality in the last two in case of
equality in the first, then for $p=1$
\[
\left\|x^r\chi_{[1,\infty)}(x) \sup_{\xi_0\leq\xi\leq x}\xi^s
\int_x^\infty y^t|f(y)|\,dy\right\|_{L^p((0,\infty),dx)}\leq
C\|f(x)\|_{L^p((0,\infty),dx)}
\]
with $C$ independent of $f$. If $t<-1$, $r+t\leq-1$ and
$r+s+t\leq-1$, then this holds for $p=\infty$.
\end{Lem}

\begin{Lem}
\label{lem:lema3} If $s<0$, $s+t\leq0$ and $r+s+t\leq-1$,with
equality holding in at most one of the last two inequalities, then
for $p=1$
\[
\left\|x^r\chi_{[1,\infty)}(x) \sup_{\xi\geq x}\xi^s \int_x^\xi
y^t|f(y)|\,dy\right\|_{L^p((0,\infty),dx)}\leq C\|f(x)\|_{L^p((0,\infty),dx)}
\]
with $C$ independent of $f$. If $s<0$, $s+t\leq-1$ and
$r+s+t\leq-1$ this holds for $p=\infty$.
\end{Lem}

\begin{Lem}
\label{lem:lema4} If $t\leq0$, $s+t\leq0$ and $r+s+t\leq-1$,with
strict inequality holding in the first two in case the third is an
equality, then for $p=1$
\[
\left\|x^r\chi_{[1,\infty)}(x) \sup_{\xi\geq x}\xi^s
\int_\xi^\infty y^t|f(y)|\,dy\right\|_{L^p((0,\infty),dx)}\leq
C\|f(x)\|_{L^p((0,\infty),dx)}
\]
with $C$ independent of $f$. If $t<-1$, $s+t\leq-1$ and
$r+s+t\leq-1$ then this holds for $p=\infty$.
\end{Lem}

\begin{Lem}
\label{lem:lema5} If $s<0$, $r+s<-1$ and $r+s+t\leq-1$, then for
$p=1$
\[
\left\|x^r\chi_{[1,\infty)}(x) \sup_{\xi\geq x}\xi^s \int_1^x
y^t|f(y)|\,dy\right\|_{L^p((0,\infty),dx)}\leq C\|f(x)\|_{L^p((0,\infty),dx)}
\]
with $C$ independent of $f$. If $s<0$, $r+s\leq 0$ and
$r+s+t\leq-1$, with equality holding in at most one of the last
two inequalities, this holds for $p=\infty$.
\end{Lem}

\begin{Lem}
\label{lem:lema6} If $r<-1$, $r+s<-1$ and $r+s+t\leq-1$, then for
$p=1$
\[
\left\|x^r\chi_{[1,\infty)}(x) \sup_{1\leq\xi\leq x}\xi^s
\int_1^{\xi} y^t|f(y)|\,dy\right\|_{L^p((0,\infty),dx)}\leq
C\|f(x)\|_{L^p((0,\infty),dx)}
\]
with $C$ independent of $f$. If $r\leq0$, $r+s\leq0$ and
$r+s+t\leq-1$, with equality in at most one of the last two
inequalities, this holds for $p=\infty$.
\end{Lem}

\subsection{Proofs of Theorem~\ref{th:main1} and Theorem~\ref{th:main2}
for regions $A_1$, $A_3$ and $A_4$}
 \label{subsec:reg1}

This section contains the proofs of the inequality
\eqref{eq:des_1} for regions $A_1$, $A_3$ and $A_4$. The results
we will prove are included in the following
\begin{Lem}
\label{lem:lema7} If $\nu>-1$, $\delta>0$, $R>0$, $j=1, 3, 4$ and
$(a, A, \nu, \delta)$ satisfy the $c_1$ conditions, then
\eqref{eq:des_1} holds for $p=1$ with $C$ independent of $f$.
\end{Lem}

\begin{Lem}
\label{lem:lema8} If $\nu>-1$, $\delta>0$, $R>0$, $j=1, 3, 4$ and
$(a, A, \nu, \delta)$ satisfy the $c_{\infty}$ conditions, then
\eqref{eq:des_1} holds for $p=\infty$ with $C$ independent of $f$.
\end{Lem}

\begin{Cor}
\label{cor:corolario2} If $1\leq p\leq\infty$, $\nu>-1$,
$\delta>0$, $R>0$, $(a, A, \nu, \delta)$ satisfy the $c_p$
conditions and $j=1,3,4$, then \eqref{eq:des_1} holds with $C$
independent of $f$.
\end{Cor}

\textbf{Proof of Corollary~\ref{cor:corolario2}}. It is enough to
observe that if $1< p<\infty$ and $(a, A, \nu, \delta)$ satisfy
the $c_p$ conditions, then $(a-1+1/p, A-1+1/p, \nu, \delta)$
satisfy the $c_1$ conditions. So, by Lemma~\ref{lem:lema7}
\begin{multline*}
\left\|\sup_{R\geq 0}\int_0^1
y^{-A+1-1/p}x^{a-1+1/p}|K_R^{\delta}(x,y)|\chi_{A_j}(x,y)|f(y)|\,
dy \right\|_{L^1((0,1),dx)}\\\leq C\|f(x)\|_{L^1((0,1),dx)},
\end{multline*}
and this is equivalent to
\[
\int_0^1 x^{a+1/p}\left(\sup_{R\geq 0}\int_0^1
|K_R^{\delta}(x,y)|\chi_{A_j}(x,y)|f(y)|\, dy\right)
\frac{dx}{x}\leq C\int_0^1 x^{A+1/p}|f(x)|\frac{dx}{x},
\]
where $j=1,3,4$. Similarly, if $(a,A,\nu,\delta)$ verify the $c_p$
conditions, then $(a+1/p, A+1/p, \nu, \delta )$ satisfy the
$c_{\infty}$ conditions. Hence, by Lemma~\ref{lem:lema8}
\begin{multline*}
\left\|x^{a+1/p}\sup_{R\geq 0}\int_0^1
|K_R^{\delta}(x,y)|\chi_{A_j}(x,y)|f(y)|\, dy\right\|_{L^{\infty}((0,1),dx)}
\\\leq C\|x^{A+1/p}f(x)\|_{L^{\infty}((0,1),dx)}.
\end{multline*}
Now, we can use the Marcinkiewicz interpolation theorem to obtain the inequality
\begin{multline*}
     \int_0^1 \left(x^{a+1/p} \left(\sup_{R\geq 0}\int_0^1
|K_R^{\delta}(x,y)|\chi_{A_j}(x,y)|f(y)|\, dy\right)\right)^p
\frac{dx}{x}\\
     \leq
C\int_0^1 \left(x^{A+1/p}|f(x)|\right)^p\frac{dx}{x},
\end{multline*}
for $1<p<\infty$ and the proof is finished.

Finally, we will prove Lemmas~\ref{lem:lema7} and~\ref{lem:lema8}
for $A_j$, $j=1, 3$ and $4$, separately.

\textbf{Proof of Lemma~\ref{lem:lema7} and Lemma~\ref{lem:lema8}
for $A_1$}. First of all, we have to note that $B_R^\delta
(f,x)=0$ when $0<R<s_1$, being $s_1$ the first positive zero of
$J_\nu$. Using the estimate~\eqref{ec:kernel}, the left side of
\eqref{eq:des_1} in this case is bounded by
\[
C\left\|x^{a+\nu+1/2}\chi_{[0,1]}(x)\sup_{s_1<R\leq
4/x}R^{2(\nu+1)} \int_0^{4/R}y^{-A+\nu+1/2}|f(y)|\,
dy\right\|_{L^p((0,1),dx)}.
\]
Making the change of variables $x=4/u$ and $y=4/v$, we have
\[
C\left\|u^{-a-\nu-\frac12-\frac2p}\chi_{[4,\infty)}(u)
\sup_{s_1\leq R \leq u}R^{2(\nu+1)}\int_R^\infty
v^{A-(\nu+\frac12)-2+\frac2p}g(v)\,dv\right\|_{L^p((0,\infty),du)},
\]
where $\|\cdot\|_{L^p((0,\infty),du)}$ denotes the $L^p$ norm in
the variable $u$, and
\[
g(v)=v^{-2/p}|f(4v^{-1})|.
\]
Note that function $g(v)$ is supported in $(1,\infty)$ and
$\|g\|_{L^p((0,\infty),du)}=\|f\|_{L^p((0,1),dx)}$. The function
$g$ will be used through the subsection, but the value $4$ may be
changed by another one, at some points, without comment. Now,
splitting the inner integral at $u$, we obtain the sum of
\begin{equation}
\label{ec:pa11}
C\left\|u^{-a-\nu-\frac12-\frac2p}\chi_{[4,\infty)}(u)
\sup_{s_1\leq R \leq u}R^{2(\nu+1)}\int_R^u
v^{A-(\nu+\frac12)-2+\frac2p}g(v)\,dv\right\|_{L^p((0,\infty),du)}
\end{equation}
and
\begin{equation}
\label{ec:pa12}
C\left\|u^{-a-\nu-\frac12-\frac2p}\chi_{[4,\infty)}(u)
\sup_{s_1\leq R \leq u}R^{2(\nu+1)}\int_u^\infty
v^{A-(\nu+\frac12)-2+\frac2p}g(v)\,dv\right\|_{L^p((0,\infty),du)}.
\end{equation}
From Lemma~\ref{lem:lema1} we get the required estimate for
\eqref{ec:pa11}, using conditions \eqref{ec:con1} and
\eqref{ec:con5}; Lemma~\ref{lem:lema2} is applied to inequality
\eqref{ec:pa12}, there we need conditions \eqref{ec:con2} and
\eqref{ec:con5} and the restriction on them. This completes the
proof of Lemmas~\ref{lem:lema7} and~\ref{lem:lema8} for $j=1$.

\textbf{Proof of Lemma~\ref{lem:lema7} and Lemma~\ref{lem:lema8}
for $A_3$}. Clearly, the left side of \eqref{eq:des_1} is bounded
by
\[
C\left\|x^a \chi_{[4/R,1]}(x)\sup_{4/x\leq R}\int_0^{x/2}y^{-A}
|K_R^{\delta}(x,y)||f(y)|\, dy\right\|_{L^p((0,1),dx)}.
\]
Splitting the inner integral at $2/R$, using the bound for the
kernel given in \eqref{ec:kernel} and the definition of
$\Phi_\nu$, we have this expression majorized by the sum of
\begin{equation}
\label{ec:pa21} \left\|x^a \chi_{[0,1]}(x)\sup_{4/x\leq
R}\int_0^{2/R}|f(y)|\frac{(Ry)^{\nu+1/2}y^{-A}}{R^{\delta}|x-y|^{\delta+1}}\,
dy\right\|_{L^p((0,1),dx)}
\end{equation}
and
\begin{equation}
\label{ec:pa22} \left\|x^a \chi_{[0,1]}(x)\sup_{4/x\leq
R}\int_{2/R}^{x/2}\frac{|f(y)|y^{-A}}{R^{\delta}|x-y|^{\delta+1}}\, dy
\right\|_{L^p((0,1),dx)}.
\end{equation}
For \eqref{ec:pa21}, taking into account that $|x-y|\simeq x$ in
$A_3$, the changes of variables $x=4/u$, $y=2/v$ give us
\[
\left\|u^{-a+(\delta+1)-\frac 2p}\chi_{[4,\infty)}(u)\sup_{u\leq
R} R^{-\delta+(\nu+1/2)}\int_R^{\infty}v^{-(\nu+1/2)+A+\frac
2p-2}g(v)\, dv\right\|_{L^p((0,\infty),du)}.
\]
Lemma~\ref{lem:lema4} can be used here. The required conditions
for $p=1$ are \eqref{ec:con2}, \eqref{ec:con4} and \eqref{ec:con5}
with the restriction in the pairs therein. For $p=\infty$ the same
inequalities are needed.

On the other hand, in \eqref{ec:pa22}, using again that $|x-y|\simeq
x$, by changing of variables $x=4/u$ and $y=2/v$ we have
\begin{multline*}
    C\left\|u^{-a+(\delta+1)-\frac
    2p}\chi_{[4,\infty)}(u)\sup_{u\leq R}R^{-\delta}\int_{2u}^R
    v^{A+\frac 2p-2}g(v)\, dv\right\|_{L^p((0,\infty),du)}\\
    \leq C\left\|u^{-a+(\delta+1)-\frac
    2p}\chi_{[4,\infty)}(u)\sup_{u\leq R}R^{-\delta}\int_u^R
    v^{A+\frac 2p-2}g(v)\, dv\right\|_{L^p((0,\infty),du)}.
\end{multline*}
Lemma~\ref{lem:lema3} can then be applied. For $p=1$, we need
$\delta>0$, which is an hypothesis, and \eqref{ec:con4} and
\eqref{ec:con5} with its corresponding restriction. For $p=\infty$
the inequalities are the same, with the requirement that
\eqref{ec:con4} is strict. This completes the proof of
Lemmas~\ref{lem:lema7} and~\ref{lem:lema8} for $j=3$.

\textbf{Proof of Lemma~\ref{lem:lema7} and Lemma~\ref{lem:lema8}
for $A_4$}. In this case, the left hand side of \eqref{eq:des_1}
is estimated by
\[
C\left\|x^a \chi_{[0,1/2]}(x)\sup_{R>4}\int_{\max(4/R,2x)}^1
y^{-A}|K_R^{\delta}(x,y)||f(y)|\, dy\right\|_{L^p((0,1),dx)}.
\]
To majorize this, we decompose the $R$-range in two regions:
$4<R\leq 2/x$ and $R\geq 2/x$. In this manner, with the bound for
the kernel given in \eqref{ec:kernel} and the definition of
$\Phi_\nu$, the previous norm is controlled by the sum of
\[
C\left\|x^a \chi_{[0,1/2]}(x)\sup_{4<R\leq 2/x}
\int_{4/R}^1 |f(y)|\frac{(Rx)^{\nu+1/2}y^{-A}}{R^{\delta}|x-y|^{\delta+1}}\, dy
\right\|_{L^p((0,1),dx)}
\]
and
\[
C\left\|x^a \chi_{[0,1/2]}(x)\sup_{R\geq 2/x}\int_{2x}^1
\frac{|f(y)|y^{-A}}{R^{\delta}|x-y|^{\delta+1}}\, dy\right\|_{L^p((0,1),dx)}.
\]
Next, using that $|x-y|\simeq y$ in $A_4$, with the changes of
variables $x=2/u$ and $y=1/v$ the previous norms are controlled by
\begin{equation}
\label{ec:pa31} C\left\|u^{-a-\frac 2p-(\nu+\frac
12)}\chi_{[4,\infty)}(u)\sup_{4<R\leq u}R^{-\delta+(\nu +\frac
12)} \int_1^{R/4}v^{A+\frac 2p
-2+(\delta+1)}g(v)\,dv\right\|_{L^p((0,\infty),du)}
\end{equation}
and
\begin{equation}
\label{ec:pa32} C\left\|u^{-a-\frac
2p}\chi_{[4,\infty)}(u)\sup_{R\geq u}R^{-\delta}\int_1^{u/4}
v^{A+\frac 2p -2+(\delta+1)}g(v)\,dv\right\|_{L^p((0,\infty),du)}.
\end{equation}
In \eqref{ec:pa31}, we use Lemma~\ref{lem:lema6}; for $p=1$,
conditions \eqref{ec:con1}, \eqref{ec:con3} and \eqref{ec:con5}
are needed; we need the same for $p=\infty$. For \eqref{ec:pa32},
Lemma~\ref{lem:lema5} requires the hypothesis $\delta>0$ and
conditions \eqref{ec:con3} and \eqref{ec:con5} for $p=1$ and the
same for $p=\infty$ with the restrictions in the pairs therein.
This proves Lemmas~\ref{lem:lema7} and~\ref{lem:lema8} for $j=4$.

\section{Proof of Theorem \ref{th:AcDebilMaxRedonda}}
\label{sec:ProofThAcDebilMaxRedonda}

Now we shall prove Theorem \ref{th:AcDebilMaxRedonda}. First note
that, by \eqref{eq:Relation}, we can write
\begin{equation*}
\mathcal{B}_R^{\delta}(f,x)
=\int_0^1f(y)\left(\frac{y}{x}\right)^{\nu+1/2}K_R^{\delta}(x,y)\,dy,
\end{equation*}
where $K_R^\delta$ is the kernel in \eqref{ec:kern}. By taking
$g(y)=f(y)y^{\nu+1/2}$, to prove the result it is enough to check
that
\[
 \int_{E}\,d\mu_\nu(x)\le \frac{C}{\lambda^{p}}
 \int_0^1|g(x)|^{p}x^{(\nu+1/2)(2-p)}\,dx,
\]
where $E=\left\{x\in(0,1):
\sup_{R>0}x^{-(\nu+1/2)}\int_0^1|g(y)||K_R^{\delta}(x,y)|\,dy>\lambda\right\}$
and $p=p_0(\delta)$. We decompose $E$ into four regions, such that
$E=\bigcup_{i=1}^{4}J_i$, where
\begin{equation*}
    J_i=\left\{x\in(0,1): \sup_{R>0}x^{-(\nu+1/2)}\int_0^{1}|g(y)|
    \chi_{B_i}(x,y)|K_R^{\delta}(x,y)|\,dy>\lambda\right\}
\end{equation*}
for $i=1,\dots,4$, with $B_1=A_1$, $B_2=A_2\cup A_5$, $B_3=A_3$,
and $B_4=A_4$ where the sets $A_i$ were defined in
\eqref{regions}. Note also that
$\int_{E}\,d\mu_\nu(x)\le\sum_{i=1}^4\int_{J_i}\,d\mu_\nu(x)$,
then we need to prove that
\begin{equation}
\label{ec:boundweak}
 \int_{J_i}\,d\mu_\nu(x)\le \frac{C}{\lambda^{p}}
 \int_0^1|g(x)|^{p}x^{(\nu+1/2)(2-p)}\,dx,
\end{equation}
for $i=1,\dots,4$ and $p=p_0(\delta)$. At some points along the
proof we will use the notation
\begin{equation}
\label{eq:integral}
I_p:=\int_0^{1}|g(y)|^{p}y^{(\nu+1/2)(2-p)}\,dy.
\end{equation}

In $J_1$, by applying \eqref{ec:kernel} and H\"older inequality
with $p=p_0$, we have
\begin{multline*}
x^{-(\nu+1/2)}\int_0^{1}|g(y)|\chi_{B_1}(x,y)
    |K_R^{\delta}(x,y)|\,dy\\
    \begin{aligned}
    &\le Cx^{-(\nu+1/2)}\int_0^{4/R}
    |g(y)|(xy)^{\nu+1/2}R^{2(\nu+1)}\,dy\\
    &\le C R^{2(\nu+1)}\left(\int_0^{4/R}
    |g(y)|^{p_0}y^{(\nu+1/2)(2-p_0)}\,dy\right)^{1/p_0}
    \left(\int_0^{4/R}y^{(2\nu+1)}\,dy\right)^{1/p'_0}\\
    &=C R^{\frac{2(\nu+1)}{p_0}}
    \left(\int_0^{4/R}|g(y)|^{p_0}y^{(\nu+1/2)(2-p_0)}\,dy\right)^{1/p_0}
    \le C R^{\frac{2(\nu+1)}{p_0}}I_{p_0}^{1/p_0}.
    \end{aligned}
\end{multline*}
Therefore,
\begin{align*}
    \sup_{R>0}x^{-(\nu+1/2)}
    \int_0^{1}|g(y)|\chi_{B_1}(x,y)|K_R^{\delta}(x,y)|\,dy
    &\le C\sup_{R>0}\chi_{[0,4/R]}(x)R^{\frac{2(\nu+1)}{p_0}}
    I_{p_0}^{1/p_0}\\&\le C
    x^{-\frac{2(\nu+1)}{p_0}}I_{p_0}^{1/p_0}.
\end{align*}
In the case $p=1$, it is clear that
\[
x^{-(\nu+1/2)}\int_0^{1}|g(y)|\chi_{B_1}(x,y)|K_R^{\delta}(x,y)|\,dy\le
C R^{2(\nu+1)}I_1
\]
and
\[
\sup_{R>0}x^{-(\nu+1/2)}
\int_0^{1}|g(y)|\chi_{B_1}(x,y)|K_R^{\delta}(x,y)|\,dy\le C
x^{-2(\nu+1)}I_1.
\]
Hence, for $p=p_0(\delta)$,
\[
    J_1\subseteq \{x\in(0,1): C x^{-\frac{2(\nu+1)}{p}}I_{p}^{1/p} >\lambda\},
\]
and this gives \eqref{ec:boundweak} for $i=1$.

In $J_3$, note first that
\begin{multline*}
    \sup_{R>0}x^{-(\nu+1/2)}
    \int_0^{1}|g(y)|\chi_{B_3}(x,y)|K_R^{\delta}(x,y)|\,dy\\
    =\sup_{R>0}x^{-(\nu+1/2)}\chi_{[4/R,1]}(x)
    \left(\int_0^{2/R}|g(y)||K_R^{\delta}(x,y)|\,dy+
    \int_{2/R}^{x/2}|g(y)||K_R^{\delta}(x,y)|\,dy\right)\\:=R_1+R_2.
\end{multline*}
For $R_1$, using \eqref{ec:kernel}, the inequality $x/2<x-y$,
which holds in $B_3$, and H\"older inequality with $p=p_0$,
\begin{align*}
    R_1
    &\le  \sup_{R>0}x^{-(\nu+3/2+\delta)}\chi_{[4/R,1]}(x)
    \int_0^{2/R}R^{\nu+1/2-\delta}y^{\nu+1/2}|g(y)|\,dy\\
    &\le\sup_{R>0}x^{-(\nu+3/2+\delta)}\chi_{[4/R,1]}(x)
    R^{\nu+1/2-\delta}R^{-\frac{2(\nu+1)}{p'_0}}I_{p_0}^{1/p_0}
    \le C x^{-\frac{2(\nu+1)}{p_0}}I_{p_0}^{1/p_0},
\end{align*}
where $I_{p_0}$ is the same as in \eqref{eq:integral}. In the case
$p=1$, the estimate $R_1\le C x^{-2(\nu+1)}I_{1}$ can be obtained
easily.

On the other hand, for $R_2$, by using \eqref{ec:kernel} and
H\"older inequality with $p=p_0$ again,
\begin{align*}
   R_2
     &\le  \sup_{R>0}x^{-(\nu+3/2+\delta)}
     \chi_{[4/R,1]}(x)I_{p_0}^{1/p_0}R^{-\delta}
     \left(\int_{2/R}^{x/2}y^{-(\nu+1/2)\frac{(2-p_0)p'_0}{p_0}}\,dy\right)^{1/p'_0}\\
    &\le \sup_{R>0}x^{-(\nu+3/2+\delta)}
    \chi_{[4/R,1]}(x)I_{p_0}^{1/p_0}R^{-\delta}
    \left(\int_{2/R}^{x/2}y^{(\nu+1/2)\frac{2-p_0}{1-p_0}}\,dy\right)^{1/p'_0}.
\end{align*}
Using that $(\nu+1/2)\frac{2-p_0}{1-p_0}<-1$ and $4/R<x<1$, we
have that
\begin{align*}
    R^{-\delta}\left(\int_{2/R}^{x/2}
    y^{(\nu+1/2)\frac{2-p_0}{1-p_0}}\,dy\right)^{1/p'_0}
    \le C\left(R^{-(\nu+1/2)\frac{2-p_0}{1-p_0}-1}\right)^{1/p'_0}
    R^{-\delta}= C
\end{align*}
and the last inequality is true because the exponent of $R$ is
zero. Then
\[
R_2\le C x^{\frac{-2(\nu+1)}{p_0}}I_{p_0}^{1/p_0}.
\]
In the case $p=1$ applying H\"older inequality, then
\begin{equation*}
     R_2\le  \sup_{R>0}x^{-(\nu+3/2+\delta)}\chi_{[4/R,1]}(x)I_1
     \,R^{-\delta}\sup_{y\in[2/R,x/2]}y^{-(\nu+1/2)}.
\end{equation*}
Now, if $\nu+1/2>0$ and $\nu+1/2<\delta$,
\begin{multline*}
\sup_{R>0}\chi_{[4/R,1]}(x)R^{-\delta}\sup_{y\in[2/R,x/2]}y^{-(\nu+1/2)}\\
=C\sup_{R>0}\chi_{[4/R,1]}(x)R^{\nu+1/2-\delta}\le
Cx^{-\nu-1/2+\delta};
\end{multline*}
and if $\nu+1/2\le0$,
\begin{multline*}
\sup_{R>0}\chi_{[4/R,1]}(x)R^{-\delta}\sup_{y\in[2/R,x/2]}y^{-(\nu+1/2)}\\
=C\sup_{R>0}\chi_{[4/R,1]}(x)R^{-\delta}x^{-(\nu+1/2)}\le
Cx^{-\nu-1/2+\delta}.
\end{multline*}
In this manner
\[
R_2\le  C x^{-2(\nu+1)}I_{1}.
\]
Therefore, collecting the estimates for $R_1$ and $R_2$ for
$p=p_0$ and $p=1$, we have shown that
\[
    J_3\subseteq \{x\in(0,1): C x^{\frac{-2(\nu+1)}{p}}(x)I^{1/p}
    >\lambda\},
\]
hence we can deduce \eqref{ec:boundweak} for $i=3$.

For the region $J_4$, we proceed as follows
\begin{align*}
    \sup_{R>0}x^{-(\nu+1/2)}&
    \int_{0}^1|g(y)|\chi_{B_4}(x,y)|K_R^\delta(x,y)|\,dy\\
    &\le \sup_{R>0}x^{-(\nu+1/2)}\chi_{[0,2/R]}(x)
    \int_{4/R}^1|g(y)||K_R^\delta(x,y)|\,dy\\
    &\kern20pt+\sup_{R>0}x^{-(\nu+1/2)}\chi_{[2/R,1]}(x)
    \int_{2x}^1|g(y)||K_R^\delta(x,y)|\,dy\\
    &\le C\sup_{R>0}x^{-(\nu+1/2)}\chi_{[0,2/R]}(x)(Rx)^{\nu+1/2}
    \int_{4/R}^1\frac{|g(y)|}{R^{\delta}|x-y|^{\delta+1}}\,dy\\
    &\kern20pt+ C\sup_{R>0}x^{-(\nu+1/2)}\chi_{[2/R,1]}(x)
    \int_{2x}^1\frac{|g(y)|}{R^{\delta}|x-y|^{\delta+1}}\,dy:=S_1+S_2.
\end{align*}
We first deal with $S_1$, we use that $y-x>y/2$, then
\begin{align*}
    S_1\le &C\sup_{R>0}\chi_{[0,2/R]}(x) R^{\nu+1/2-\delta}
    \int_{4/R}^1\frac{|g(y)|}{y^{\delta+1}}\,dy\\
    &\le C\sup_{R>0}\chi_{[0,2/R]}(x) R^{\nu+1}\int_{4/R}^1\frac{|g(y)|}{\sqrt{y}}\,dy
    \le C x^{-(\nu+1)}\int_x^1\frac{|g(y)|}{\sqrt{y}}\,dy.
\end{align*}
Now for $p=p_0$ or $p=1$, we have that $2\nu+1-p(\nu+1)>-1$ and
Hardy's inequality \cite[Lemma 3.14, p. 196]{SteinWeiss} is
applied in the following estimate
\begin{align*}
    \int_0^1|S_1(x)|^{p}x^{2\nu+1}\,dx& \le C
    \int_0^1\left(\int_x^1\frac{|g(y)|}{\sqrt{y}}\,dy\right)^{p}
    x^{2\nu+1-p(\nu+1)}\,dx\\
    &\le C \int_0^1\left|\frac{g(y)}{\sqrt y}\right|^{p}y^{2\nu+1-p\nu}\,dy
    =C\int_0^1|g(y)|^{p}y^{(\nu+1/2)(2-p)}\,dy.
\end{align*}
Concerning $S_2$, observe that
$\sup_{R>0}\chi_{[2/R,1]}(x)R^{-\delta}\le Cx^{\delta}$, thus
\[
S_2\le C
x^{-\nu-1/2+\delta}\int_x^1\frac{|g(y)|}{y^{\delta+1}}\,dy.
\]
Since for $p=p_0$ or $p=1$ we have that
$2\nu+1-p(\nu+1/2-\delta)>-1$, we can use again Hardy's inequality
to complete the required estimate. Indeed,
\begin{align*}
    \int_0^1|S_2(x)|^{p}x^{2\nu+1}\,dx&
    \le C\int_0^1\left(\int_x^1\frac{|g(y)|}{y^{\delta+1}}\,dy\right)^{p}
    x^{2\nu+1-p(\nu+1/2-\delta)}\,dx\\
    &\le C\int_0^1\left|\frac{g(y)}{y^{\delta+1}}\right|^{p}
    y^{2\nu+1-p(\nu+1/2-\delta)+p}\,dy\\&
    =C\int_0^1|g(y)|^{p}y^{(\nu+1/2)(2-p)}\,dy.
\end{align*}
With the inequalities for $S_1$ and $S_2$, we can conclude
\eqref{ec:boundweak} for $i=4$.

To prove \eqref{ec:boundweak} for $i=2$ we define, for $k$ a
nonnegative integer, the intervals
\[
I_k=[2^{-k-1},2^{-k}], \qquad N_k=[2^{-k-3},2^{-k+2}]
\]
and the function $g_k(y)=|g(y)|\chi_{I_k}(y)$. By using
\eqref{ec:kernel} for $x/2<y<2x$, with $x\in (0,1)$, we have the
bound
\[
|K_R^\delta (x,y)|\le \frac{C}{R^{\delta}(|x-y|+2/R)^{\delta+1}}.
\]
Then
\[
J_{2}\subset \left\{x\in (0,1): \sup_{R>0}\sum_{k=0}^\infty
\int_{x/2}^{\min{\{2x,1\}}}
\frac{g_k(t)}{R^{\delta}(|x-y|+2/R)^{\delta+1}}\, dy> C \lambda
x^{\nu+1/2}\right\}.
\]
Since at most three of these integrals are not zero for each $x\in
(0,1)$
\begin{align*}
J_2&\subset \bigcup_{k=0}^\infty \left\{x\in (0,1):
3\sup_{R>0}\int_{x/2}^{\min{\{2x,1\}}}
\frac{g_k(t)}{R^{\delta}(|x-y|+2/R)^{\delta+1}}\, dy> C \lambda
x^{\nu+1/2}\right\}\\
&\subset \bigcup_{k=0}^\infty \left\{x\in N_k : M(g_k,x)> C
\lambda x^{\nu+1/2}\right\}
\end{align*}
where in the las step we have used that
\[
\sup_{R>0}\int_{x/2}^{\min{\{2x,1\}}}
\frac{g_k(t)}{R^{\delta}(|x-y|+2/R)^{\delta+1}}\, dy\le C
M(g_k,x).
\]
By using the estimate $x\simeq 2^{-k}$ for $x\in N_k$, we can
check easily that
\[
J_2\subset \bigcup_{k=1}^\infty \left\{x \in N_k : M(g_k,x)> C
\lambda 2^{-k(\nu+1/2)}\right\}.
\]
Finally by using again that $x\simeq 2^{-k}$ for $x\in I_k, N_k$ and the
weak type norm inequality for the Hardy-Littlewood maximal
function we have
\begin{align*}
\int_{J_2}x^{2\nu+1}\, dx &\le C \sum_{k=0}^\infty 2^{-k(2\nu+1)}
\int_{\left\{x\in  N_k : M(g_k,x)> C \lambda
2^{-k(\nu+1/2)}\right\}}\, dx\\&\le C \sum_{k=0}^\infty
\frac{2^{pk(\nu+1/2)-k(2\nu+1)}}{\lambda^p}\int_{I_k}|g(y)|^p\,
dy\\&\le \frac{C}{\lambda^p}\int_0^1 |g(y)|^p y^{(\nu+1/2)(2-p)}\,
dy
\end{align*}
and the proof is complete.
\section{Proof of Theorem \ref{th:AcDebilRestMaxRedonda}}
\label{sec:acdelrest}
To conclude the result we have to prove \eqref{ec:boundweak} with
$g(x)=\chi_E(x)$ and $p=p_1$. For $J_1$ and $J_2$ the result
follows by using the steps given in the proof of Theorem
\ref{th:AcDebilMaxRedonda} for the same intervals. To analyze
$J_3$ we proceed as we did for $J_4$ in the proof of Theorem
\ref{th:AcDebilMaxRedonda}. In this case we obtain that
\begin{multline*}
\sup_{R>0}x^{-(\nu+1/2)}\int_0^1
|g(y)|\chi_{B_3}(x,y)|K_r^\delta(x,y)|\\\le
C\left(x^{-(\nu+1)}\int_0^x \frac{|g(y)|}{\sqrt{y}}\, dy+
x^{-(\nu+3/2+\delta)}\int_0^x |g(y)| y^\delta\, dy\right).
\end{multline*}
Now taking into account that for $p=p_1$ we have
$2\nu+1-p(\nu+1)<-1$ and $2\nu+1-p(\nu+3/2+\delta)<-1$ we can
apply Hardy's inequalities to obtain that
\[
\int_0^1\left(x^{-(\nu+1)}\int_0^x \frac{|g(y)|}{\sqrt{y}}\,
dy\right)^{p}x^{2\nu+1}\, dx\le C \int_0^1 |g(y)|^p
y^{(\nu+1/2)(2-p)}\, dy
\]
and
\[
\int_0^1\left(x^{-(\nu+3/2+\delta)}\int_0^x |g(y)| y^\delta\,
dy\right)^{p}x^{2\nu+1}\, dx\le C \int_0^1 |g(y)|^p
y^{(\nu+1/2)(2-p)}\, dy,
\]
with these two inequalities we can deduce that
\eqref{ec:boundweak} holds for $J_3$ with $p=p_1$ in this case.

The main difference with the previous proof appears in the
analysis of $J_4$. To deal with this case, we have to use the
following lemma \cite[Lemma 16.5]{Ch-Muc}
\begin{Lem}
\label{lem:Muck} If $1<p<\infty$, $a>-1$, and $E\subset
[0,\infty)$, then
\[
\left(\int_{E}x^a\, dx\right)^p\le
2^p(a+1)^{1-p}\int_{E}x^{(a+1)p-1}\, dx.
\]
\end{Lem}
In this case, it is enough to prove that
\[
\int_{\mathcal{J}}\, d\mu_\nu(x)\le \frac{C}{\lambda^p}\int_{0}^1
\chi_E(y)\,d\mu_\nu(y),
\]
where
\begin{equation*}
    \mathcal{J}=\left\{x\in(0,1):
    \sup_{R>0}x^{-(\nu+1/2)}\int_0^{1}\chi_E(y)
    \chi_{B_4}(x,y)y^{\nu+1/2}|K_R^{\delta}(x,y)|\,dy>\lambda\right\},
\end{equation*}
and this can be deduced immediately by using the inclusion
\begin{equation}
\label{ec:final}
\mathcal{J}\subseteq [0,\min\{1,H\}]
\end{equation}
with
\[
H^{2(\nu+1)}=\frac{C}{\lambda^p}\int_{0}^1 \chi_E(y)\,d\mu_\nu(y).
\]
Let's prove \eqref{ec:final}. By using \eqref{ec:kern} and the
estimate $y-x>y/2$, we have
\begin{multline*}
\sup_{R>0}x^{-(\nu+1/2)}\int_0^{1}\chi_E(y)
\chi_{B_4}(x,y)y^{\nu+1/2}|K_R^{\delta}(x,y)|\,dy
\\
\le C\sup_{R>0}
R^{-\delta+\nu+1/2}\chi_{[0,2/R]}(x)\int_{4/R}^1\chi_E(y)
y^{-\delta+\nu-1/2}\, dy\\
+ C\sup_{R>0}
R^{-\delta}x^{-(\nu+1/2)}\chi_{[2/R,1]}(x)\int_{2x}^1 \chi_E(y)
y^{-\delta+\nu-1/2}\, dy.
\end{multline*}
In the first summand we can use that $R^{-\delta+\nu+1/2}\le C
x^{\delta-\nu-1/2}$ and in the second one that $R^{-\delta}\le
x^{\delta}$. Moreover observing that with $p=p_1$ it holds
$-\delta+\nu+1/2=2(\nu+1)/p$ we obtain that
\begin{align*}
\sup_{R>0}x^{-(\nu+1/2)}\int_0^{1}\chi_E(y)\chi_{B_4}y^{\nu+1/2}|K_R^{\delta}(x,y)|\,dy&\le
C x^{-2(\nu+1)/p}\int_{E}y^{-1+2(\nu+1)/p}\, dy\\
&\le C x^{-2(\nu+1)/p}\int_{E}\,d\mu_\nu(y),
\end{align*}
where in the last step we have used Lemma \ref{lem:Muck}, and this
is enough to deduce the inclusion in \eqref{ec:final}.

\section{Proofs of Theorem \ref{th:noweak} and Theorem \ref{th:nostrong}}
\label{sec:negativeths}

This section will be devoted to the proofs of Theorem
\ref{th:noweak} and Theorem \ref{th:nostrong}. To this end we need
a suitable identity for the kernel and in order to do that we have
to introduce some notation. $H_{\nu}^{(1)}$ will denote the Hankel
function of the first kind, and it is defined as follows
\[
H_{\nu}^{(1)}(z)=J_{\nu}(z)+iY_{\nu}(z),
\]
where $Y_{\nu}$ denotes the Weber's function, given by
\begin{equation*}
Y_{\nu}(z)=\frac{\Bes(z)\cos \nu \pi-J_{-\nu}(z)}{\sin \nu
\pi},\,\,  \nu\notin \mathbb{Z}, \text{ and } Y_n(z)=\lim_{\nu\to
n}\frac{\Bes(z)\cos \nu \pi-J_{-\nu}(z)}{\sin \nu \pi}.
\end{equation*}
From these definitions, we have
\begin{equation*}
H_{\nu}^{(1)}(z)=\frac{J_{-\nu}(z)-e^{-\nu \pi i}\Bes(z)}{i\sin
\nu
\pi}, \,\,  \nu\notin \mathbb{Z},\\
\text{ and } H_n^{(1)}(z)=\lim_{\nu\to n}\frac{J_{-\nu}(z)-e^{-\nu
\pi i}\Bes(z)}{i\sin \nu \pi}.
\end{equation*}
For the function $H_\nu^{(1)}$, the asymptotic
\begin{equation}
\label{inftyH} H_{\nu}^{(1)}(z)=\sqrt{\frac{2}{\pi
z}}e^{i(z-\nu\pi/2-\pi/4)}[A+O(z^{-1})], \quad |z|>1,\quad -\pi <
\arg(z)<2\pi,
\end{equation}
holds for some constant $A$.

In \cite[Lemma 1]{Ci-Ro} the following lemma was proved
\begin{Lem}
\label{lem:expresnucleo} For $R>0$ the following holds:
\[K_R^\delta(x,y)=I_{R,1}^\delta(x,y)+I_{R,2}^\delta(x,y)\]
with
\[
I_{R,1}^{\delta}(x,y)=(xy)^{1/2}\int_0^{R}z\multi\Bes(zx)\Bes(zy)\,
dz
\]
and
\[ I_{R,2}^{\delta}(x,y)=\lim_{\varepsilon\to 0}
\frac{(xy)^{1/2}}{2}\int_{\mathbf{S_\varepsilon}}
\multi \frac{z H^{(1)}_{\nu}(z)\Bes(zx)\Bes(zy)}{\Bes(z)}\,dz,
\]
where, for each $\varepsilon>0$, $\mathbf{S_\varepsilon}$ is the
path of integration given by the interval
$R+i[\varepsilon,\infty)$ in the direction of increasing imaginary
part and the interval $-R+i[\varepsilon,\infty)$ in the opposite
direction.
\end{Lem}

Then, by Lemma \ref{lem:expresnucleo} we have
\[
\mathcal{K}_R^\delta(x,y)=\mathcal{I}_{R,1}^\delta(x,y)+\mathcal{I}_{R,2}^\delta(x,y)
\]
where
$\mathcal{I}_{R,j}^\delta(x,y)=(xy)^{-(\nu+1/2)}I_{R,j}^{\delta}(x,y)$
for $j=1,2$. The main tool to deduce our negative results will be
the following lemma

\begin{Lem}
\label{lem:zero} For $\nu>-1/2$, $\delta>0$, and $R>0$ it is
verified that
\[
\mathcal{K}_R^\delta(0,y)=\frac{2^{\delta-\nu}\Gamma(\delta+1)}{\Gamma(\nu+1)}R^{2(\nu+1)}
\frac{J_{\nu+\delta+1}(yR)}{(yR)^{\nu+\delta+1}}+\mathcal{I}_{R,2}^\delta(0,y),
\]
where
\begin{equation}
\label{ec:boundI} \left|\mathcal{I}_{R,2}^\delta(0,y)\right|\le
C\begin{cases} R^{2\nu-\delta+1}, &
yR\le 1,\\
R^{\nu-\delta+1/2}y^{-(\nu+1/2)}, & yR>1.
\end{cases}
\end{equation}
\end{Lem}
\begin{proof}
From \eqref{zero}, it is clear that
\[
\mathcal{I}_{R,1}^\delta(0,y)=\frac{y^{-\nu}}{2^\nu\Gamma(\nu+1)}\int_0^R
z^{\nu+1}\multi\Bes(zy)\, dz.
\]
Now, by using Sonine's identity \cite[Ch. 12, 12.11, p. 373]{Wat}
\[
\int_0^1 s^{\nu+1}\left(1-s^2\right)^\delta\Bes(sy)\,
ds=2^{\delta}\Gamma(\delta+1)\frac{J_{\nu+\delta+1}(y)}{y^{\delta+1}},
\qquad \nu,\delta>-1,
\]
we deduce the leading term of the expression for
$\mathcal{K}_{R}^\delta(0,y)$.

To control the term
\[
\mathcal{I}_{R,2}^\delta(0,y)=\lim_{\varepsilon\to
0}\frac{y^{-(\nu+1/2)}}{2}\int_{\mathbf{S_\varepsilon}} \multi
\frac{z^{\nu+1/2}
H^{(1)}_{\nu}(z)(zy)^{1/2}\Bes(zy)}{\Bes(z)}\,dz,
\]
we start by using the asymptotic expansions given in \eqref{inftyH}
and \eqref{infty} for $H_{\nu}^{(1)}(z)$ and $\Bes(z)$. We see
that on $\mathbf{S_\varepsilon}$, the path of integration
described in Lemma \ref{lem:expresnucleo}, for $t=\Im(z)$ the
estimate
\[
\left|\frac{H_{\nu}(z)}{\Bes(z)}\right|\leq C e^{-2t},
\]
holds for $t>0$. Now, from  \eqref{zero} and \eqref{infty}, it is
clear that for $z=\pm R+it$
\[
|\sqrt{zy}J_{\nu}(zy)|\le Ce^{yt}\Phi_\nu((R+t)y)
\]
where $\Phi_\nu$ is the function in \eqref{ec:aux}. Then
\[
|\mathcal{I}_{R,2}^\delta(0,y)|\le C
y^{-(\nu+1/2)}R^{-2\delta}\int_0^\infty
t^{\delta}(R+t)^{\nu+\delta+1/2}\Phi_\nu((R+t)y)e^{-(2-y)t}\, dt.
\]
If $y>1/R$ we have the inequality $\Phi_\nu((R+t)y)\le C$, then
\begin{align*}
|\mathcal{I}_{R,2}^\delta(0,y)| &\le C y^{-(\nu+1/2)} R^{-2\delta}
\int_0^\infty t^{\delta}(R+t)^{\nu+\delta+1/2}
e^{-(2-y)t}\, dt\\
&\le C y^{-(\nu+1/2)}R^{-\delta}(R^{\nu+1/2}+R^{-\delta})\le C
R^{\nu-\delta+1/2}y^{-(\nu+1/2)}
\end{align*}
and \eqref{ec:boundI} follows in this case. If $y\le 1/R$ we
obtain the bound in \eqref{ec:boundI} with the estimate
$\Phi_\nu((R+t)y)\le C (\Phi_\nu(yR)+(yt)^{\nu+1/2})$. Indeed,
\begin{multline*}
|\mathcal{I}_{R,2}^\delta(0,y)| \le C y^{-(\nu+1/2)} R^{-2\delta}
\Phi_\nu(yR) \int_0^\infty t^{\delta}(R+t)^{\nu+\delta+1/2}
e^{-(2-y)t}\, dt\\+ C R^{-2\delta} \int_0^\infty
t^{\nu+\delta+1/2}(R+t)^{\nu+\delta+1/2} e^{-(2-y)t}\, dt\\\le C
(R^{2\nu-\delta+1}+R^{\nu-2\delta+1/2}+R^{\nu-\delta+1/2}+R^{-2\delta})\le
R^{2\nu-\delta+1}.
\end{multline*}
\end{proof}

\begin{Lem}
\label{lem:cota0} For $\nu>-1/2$ and $0<\delta\le \nu+1/2$, the
estimate
\[
\|\mathcal{K}_R^\delta(0,y)\|_{L^{p_0}((0,1),d\mu_\nu)}\ge C
R^{\nu-\delta+1/2}(\log R)^{1/p_0}
\]
holds.
\end{Lem}
\begin{proof}
We will use the decomposition in Lemma \ref{lem:zero}. By using
\eqref{zero} and \eqref{infty} as was done in \cite[Lemma
2.1]{Ci-RoWave} we obtain that
\[
\left\|R^{2(\nu+1)}
\frac{J_{\nu+\delta+1}(yR)}{(yR)^{\nu+\delta+1}}\right\|_{L^{p_0}((0,1),d\mu_\nu)}\ge
C R^{\nu-\delta+1/2}(\log R)^{1/p_0}.
\]
With the bound \eqref{ec:boundI} it can be deduced that
\[
\left\|\mathcal{I}_{R,2}^\delta(0,y)\right\|_{L^{p_0}((0,1),d\mu_\nu)}\le
C R^{\nu-\delta+1/2}.
\]
With the previous estimates the proof is completed.
\end{proof}

Finally, the last element that we need to prove Theorems
\ref{th:noweak} and \ref{th:nostrong} is the norm inequality for
 finite linear combinations of the functions $\{\psi_j\}_{j\ge 1}$
contained in the next lemma. Its proof is long and technical and
it will be done in the last section.

\begin{Lem}
\label{lem:pol} For $\nu>-1/2$, $R>0$, $1<p<\infty$ and $f$ a
linear combination of the functions $\{\psi_j\}_{1\le j\le N(R)}$
with $N(R)$ a positive integer such that $N(R)\simeq R$, the inequality
\[
\|f\|_{L^\infty ((0,1),d\mu_\nu)}\le C
R^{2(\nu+1)/p}\|f\|_{L^{p,\infty} ((0,1),d\mu_\nu)}
\]
holds.
\end{Lem}

\begin{proof}[Proof of Theorem \ref{th:noweak}]
With the bound in Lemma \ref{lem:cota0} we have
\begin{align*}
(\log R)^{1/p_0}&\le C R^{-2(\nu+1)/p_1}
\left\|\mathcal{K}_{R}^\delta(0,y)\right\|_{L^{p_0}((0,1),d\mu_\nu)}\\
& = C R^{-2(\nu+1)/p_1}
\sup_{\|f\|_{L^{p_1}((0,1),d\mu_\nu)}=1}\left|\int_0^1
\mathcal{K}_{R}^\delta(0,y) f(y)\, d\mu_\nu\right|\\
& = C R^{-2(\nu+1)/p_1}
\sup_{\|f\|_{L^{p_1}((0,1),d\mu_\nu)}=1}\left|\mathcal{B}_{R}^\delta
f(0)\right|.
\end{align*}
From the previous estimate the result for
$\delta=\nu+1/2$ follows. In the case $\delta<\nu+1/2$ it is
obtained by using Lemma \ref{lem:pol} because
\begin{multline*}
R^{-2(\nu+1)/p_1}
\sup_{\|f\|_{L^{p_1}((0,1),d\mu_\nu)}=1}\left|\mathcal{B}_{R}^\delta
f(0)\right|\\\le C
\sup_{\|f\|_{L^{p_1}((0,1),d\mu_\nu)}=1}\left\|\mathcal{B}_{R}^\delta
f(x)\right\|_{L^{p_1,\infty}((0,1),d\mu_\nu)}
\end{multline*}
since $\mathcal{B}_{R}^\delta f(x)$ is a linear combination of the
functions $\{\psi_j\}_{1\le j\le N(R)}$ with $N(R)\simeq R$.
\end{proof}

\begin{proof}[Proof of Theorem \ref{th:nostrong}]
In the case $\delta <\nu+1/2$, the result follows from Theorem
\ref{th:noweak} by using a duality argument. Indeed, it is clear
that
\begin{align}
\sup_{E\subset
(0,1)}\frac{\|\mathcal{B}^{\delta}_{R}\chi_E\|_{L^{p_0}((0,1),d\mu_\nu)}}
{\|\chi_E\|_{L^{p_0}((0,1),d\mu_\nu)}} &=\sup_{E\subset (0,1)}
\sup_{\|f\|_{L^{p_1}((0,1),d\mu_\nu)}=1}
\frac{\left|\int_0^1f(y)\mathcal{B}^{\delta}_{R}\chi_E(y)\,
d\mu_\nu\right|}
{\|\chi_E\|_{L^{p_0}((0,1),d\mu_\nu)}}\notag\\
&= \sup_{\|f\|_{L^{p_1}((0,1),d\mu_\nu)}=1} \sup_{E\subset (0,1)}
\frac{\left|\int_0^1\chi_E(y)\mathcal{B}^{\delta}_{R}f(y)\,
d\mu_\nu\right|}
{\|\chi_E\|_{L^{p_0}((0,1),d\mu_\nu)}}\label{ec:lambdacero}.
\end{align}
By Theorem \ref{th:noweak} it is possible to choose a function $g$
such that $\|g\|_{L^{p_1}((0,1),d\mu_\nu)}=1$ and
\[
\|\mathcal{B}_{R}^\delta
g(x)\|_{L^{p_1,\infty}((0,1),d\mu_\nu)}\ge C (\log R)^{1/p_0}.
\]
Then, with the notation
\[
\mu_\nu(E)=\int_{E}\, d\mu_\nu,
\]
we have
\begin{equation}
\label{ec:lambda} \lambda^{p_1}\mu_\nu(A)\ge C (\log R)^{p_1/p_0},
\end{equation}
for some positive $\lambda$ and $A=\{x\in (0,1): |B_{R}^\delta
g(x)|>\lambda\}$. Now, we consider the subsets of $A$
\[
A_1=\{x\in (0,1): B_{R}^\delta g(x)>\lambda\} \qquad\text{ and }
\qquad A_2=\{x\in (0,1): B_{R}^\delta g(x)<-\lambda\}
\]
and we define $D=A_1$ if $\mu_\nu(A_1)\ge \mu_\nu(A)/2$ and
$D=A_2$ otherwise. Then, by \eqref{ec:lambda}, we deduce that
\begin{equation}
\label{ec:lambda2} \lambda \ge C \frac{(\log
R)^{1/p_0}}{\mu_\nu(D)^{1/p_1}}.
\end{equation}
Taking $f=g$ and $E=D$ in \eqref{ec:lambdacero} and using
\eqref{ec:lambda2}, we see that
\[
\sup_{E\subset
(0,1)}\frac{\|\mathcal{B}^{\delta}_{R}\chi_E\|_{L^{p_0}((0,1),d\mu_\nu)}}
{\|\chi_E\|_{L^{p_0}((0,1),d\mu_\nu)}} \ge C
\lambda\frac{\mu_\nu(D)}{\|\chi_D\|_{L^{p_0}((0,1),d\mu_\nu)}} \ge
C (\log R)^{1/p_0}
\]
and the proof is complete in this case. For $\delta=\nu+1/2$ the
result follows from Theorem \ref{th:noweak} with a standard
duality argument.
\end{proof}

\section{Proof of Lemma \ref{lem:pol}}
\label{sec:techlemma}
To proceed with the proof of Lemma \ref{lem:pol} we need some
auxiliary results that are included in this section.

We start by defining a new operator. For each non-negative integer
$r$, we consider the vector of coefficients
$\alpha=(\alpha_1,\dots,\alpha_{r+1})$ and we define
\[
T_{r,R,\alpha}f(x)=\sum_{\ell=1}^{r+1}\alpha_\ell
\mathcal{B}_{\ell R}^{r}f(x).
\]
This new operator is an analogous of the \textit{generalized
delayed means} considered in \cite{SteinDuke}. In \cite{SteinDuke}
the operator is defined in terms of the Ces\`{a}ro means instead
of the Bochner-Riesz means. The properties of $T_{r,R,\alpha}$
that we need are summarized in the next lemma

\begin{Lem}
\label{lem:delay} For each non-negative integer $r$ and $\nu\ge
-1/2$, the following statements hold
\begin{enumerate}
\item[a)] $T_{r,R,\alpha}f$ is a linear combination of the
functions $\{\psi_j\}_{1\le j\le N((r+1)R)}$, where $N((r+1)R)$ is
a non-negative integer such that $N((r+1)R)\simeq (r+1)R$;

\item[b)] there exists a vector of coefficients $\alpha$,
verifying that $|\alpha_\ell|\le A$, for $\ell=1,\dots, r+1$, with
$A$ independent of $R$ and such that $T_{r,R,\alpha}f(x)=f(x)$ for
each linear combination of the functions $\{\psi_j\}_{1\le j\le
N(R)}$ where $N(R)$ is a positive integer. Moreover, in this case,
for $r>\nu+1/2$,
\[
\|Tf_{r,R,\alpha}\|_{L^1 ((0,1),d\mu_\nu)}\le
C\|f\|_{L^1((0,1),d\mu_\nu)}\] and
\[\|T_{r,R,\alpha}f\|_{L^\infty ((0,1),d\mu_\nu)}\le C
\|f\|_{L^\infty((0,1),d\mu_\nu)},
\]
with $C$ independent of $R$ and $f$.
\end{enumerate}
\end{Lem}
\begin{proof}
Part a) is a consequence of the definition of $T_{r,R,\alpha}$ and
the fact that the $m$-th zero of the Bessel function $J_\nu$, with
$\nu \ge-1/2$, is contained in the interval
$(m\pi+\nu\pi/2+\pi/2,m\pi+\nu\pi/2+3\pi/4)$.

To prove b) we consider $f(x)=\sum_{j=1}^{N(R)} a_j \psi_j(x)$. In
order to obtain the vector of coefficients such that
$T_{r,R,\alpha}f(x)=f(x)$ the equations
\[
\sum_{\ell=1}^{r+1}\alpha_\ell \left(1-\frac{s_{k}^2}{(\ell
R)^2}\right)^r=1,
\]
for all $k=1,\dots,N(R)$, should be verified. After some elementary manipulations each
one of the previous equations can be written as
\[
\sum_{j=0}^r
s_k^{2j}\binom{r}{j}\frac{(-1)^j}{R^{2j}}\sum_{\ell=1}^{r+1}
\frac{\alpha_\ell}{\ell^{2j}}=1
\]
and this can be considered as a polynomial in $s_k^2$ which must
be equal $1$, therefore we have the system of equations
\[
\sum_{\ell=1}^{r+1}\frac{\alpha_\ell}{\ell^{2j}}=\delta_{j,0},
\qquad j=0,\dots,r.
\]
This system has an unique solution because the determinant of the
matrix of coefficients is a Vandermonde's one. Of course for each
$\ell=1,\dots,r+1$, it is verified that $|\alpha_\ell|\le A$, with
$A$ a constant depending on $r$ but not on $N(R)$.

The norm estimates are consequence of the uniform boundedness
\[
\|\mathcal{B}_R^\delta f\|_{L^p((0,1),d\mu_\nu)}\le C
\|f\|_{L^p((0,1),d\mu_\nu)},
\]
for $p=1$ and $p=\infty$ when $\delta > \nu+1/2$ (see
\cite{Ci-Ro}).
\end{proof}

In the next lemma we will control the $L^\infty$-norm of a finite linear
combination of the functions $\{\psi_j\}_{j\ge 1}$ by its
$L^1$-norm.

\begin{Lem}
\label{lem:infty1} If $\nu>-1/2$ and $f(x)$ is a linear
combination of the functions $\{\psi_j\}_{1\le j\le N(R)}$ with $N(R)$ a positive integer such that
$N(R)\simeq R$, the inequality
\[
\|f\|_{L^\infty ((0,1),d\mu_\nu)}\le C
R^{2(\nu+1)}\|f\|_{L^1((0,1),d\mu_\nu)}
\]
holds.
\end{Lem}
\begin{proof}
It is clear that
\[
f(x)=\sum_{j=1}^{N(R)} \psi_j(x)\int_0^1 f(y) \psi_j(y)\,
d\mu_\nu(y).
\]
Now, using H\"{o}lder inequality and Lemma \ref{Lem:NormaFunc} we
have
\begin{align*}
\|f\|_{L^\infty ((0,1),d\mu_\nu)}&\le C\sum_{j=1}^{N(R)}
\|\psi_j\|_{L^\infty
((0,1),d\mu_\nu)}^2\|f\|_{L^{1}((0,1),d\mu_\nu)}\\&\le C
\|f\|_{L^{1}((0,1),d\mu_\nu)} \sum_{j=1}^{N(R)}j^{2\nu+1}\le C
R^{2(\nu+1)}\|f\|_{L^1((0,1),d\mu_\nu)}.
\end{align*}
\end{proof}

The following lemma is a version in the space $((0,1),d\mu_\nu)$
of Lemma 19.1 in \cite{Ch-Muc}. The proof can be done in the same
way, with the appropriate changes, so we omit it.

\begin{Lem}
\label{lem:fuerdeb} Let $\nu>-1$, $1<p<\infty$ and $T$ be a linear
operator defined for functions in $L^1((0,1),d\mu_\nu)$ and such
that
\[
\|Tf\|_{L^\infty ((0,1),d\mu_\nu)}\le A
\|f\|_{L^1((0,1),d\mu_\nu)} \,\text{ and }\,\|Tf\|_{L^\infty
((0,1),d\mu_\nu)}\le B \|f\|_{L^\infty((0,1),d\mu_\nu)},
\]
then
\[
\|Tf\|_{L^\infty ((0,1),d\mu_\nu)}\le C
A^{1/p}B^{1/p'}\|f\|_{L^{p,\infty}((0,1),d\mu_\nu)}.
\]
\end{Lem}

Now, we are prepared to conclude the proof of Lemma \ref{lem:pol}.

\begin{proof}[Proof of Lemma \ref{lem:pol}]
We consider the operator $T_{r,R,\alpha}f$ given in Lemma
\ref{lem:delay} b) with $r>\nu+1/2$. By Lemma \ref{lem:delay} and
Lemma \ref{lem:infty1} we have
\begin{align*}
\|T_{r,R,\alpha}f\|_{L^\infty((0,1),d\mu_\nu)}&\le C
((r+1)R)^{2(\nu+1)}
\|T_{r,R,\alpha}f\|_{L^1((0,1),d\mu_\nu)}\\&\le C
R^{2(\nu+1)}\|f\|_{L^1((0,1),d\mu_\nu)}.
\end{align*}
From b) in Lemma \ref{lem:delay} we obtain the estimate
\[
\|T_{r,R,\alpha}f\|_{L^\infty((0,1),d\mu_\nu)}\le C
\|f\|_{L^\infty((0,1),d\mu_\nu)}.
\]
So, by using Lemma \ref{lem:fuerdeb}, we obtain the inequality
\[
\|T_{r,R,\alpha}f\|_{L^\infty((0,1),d\mu_\nu)}\le C
 R^{2(\nu+1)/p}\|f\|_{L^{p,\infty}((0,1),d\mu_\nu)}
\]
for any $f\in L^1((0,1),d\mu_{\nu})$. Now, since
$T_{r,R,\alpha}f(x)=f(x)$ for a linear combination of the
functions $\{\psi_j\}_{1\le j\le N(R)}$, the proof is complete.
\end{proof}

\end{document}